\documentclass{amsart}
\usepackage{amssymb}
\usepackage{amsmath}
\usepackage{amsfonts}

\begin{document}

\newtheorem{theo}{Theorem}[section]
\newtheorem{atheo}{Theorem*}
\newtheorem{prop}[theo]{Proposition}
\newtheorem{aprop}[atheo]{Proposition*}
\newtheorem{lemma}[theo]{Lemma}
\newtheorem{alemma}[atheo]{Lemma*}
\newtheorem{exam}[theo]{Example}
\newtheorem{coro}[theo]{Corollary}
\theoremstyle{definition}
\newtheorem{defi}[theo]{Definition}
\newtheorem{rem}[theo]{Remark}

\title[Tarskian truth hierarchies]
{Hierarchies of Tarskian truth predicates}
\author{Nik Weaver}
\address{Department of Mathematics, Washington University in St. Louis}
\email{nweaver@wustl.edu}
\date{January 8, 2022}


\begin{abstract}
Hierarchies of non self-applicative truth predicates take us beyond $\Gamma_0$.
\end{abstract}

\maketitle

\section{Introduction}

The purpose of this note is to make some earlier work I did on hierarchies of truth theories
\cite{weaver} more accessible. The immediate prompt for doing this was the discovery of an
error in the proof of an important lemma in \cite{weaver}, which was pointed out to me by
Asger Ipsen \cite{ipsen} (though the statement of the lemma is still correct). But in hindsight
I also see that the whole presentation was, perhaps, overly technical and probably not very readable.
It seemed worthwhile to summarize the ideas at a more conceptual level, suppressing some of the
coding technicalities which one needs in order to convert the high-level ideas I will describe
here into the machine language of a formal proof. The proofs have also been simplified a little.

I will be working with number-theoretic formal systems, and the easiest way to
implement a truth predicate for such a system is to introduce a predicate symbol $T$ which
takes a natural number as its argument, with the intended meaning of $T(n)$ being ``the
formula with G\"odel number $n$ is true''. G\"odel numbers are an example of the sort of
coding technicality I intend to pass over in this paper. In fact, more is needed; for
instance, we need to be able within the system to perform basic syntactic constructions on
numerically encoded strings. E.g., if $\ulcorner A\urcorner$ and $\ulcorner B\urcorner$ are the
codes for two formulas $A$ and $B$, and $\ulcorner A \wedge B \urcorner$ is the code for their
conjunction, then we require the ability to calculate $\ulcorner A \wedge B\urcorner$ from $\ulcorner A\urcorner$
and $\ulcorner B\urcorner$. More precisely, we must be able to express and reason about a numerical function
of two arguments which yields $\ulcorner A\wedge B\urcorner$ when given the inputs
$\ulcorner A\urcorner$ and $\ulcorner B\urcorner$. Probably primitive recursive arithmetic
would suffice for everything we need to do along these lines. These are the kinds of details I will
either relegate to footnotes or omit entirely from the present account.

The truth predicates I will be discussing are never self-applicative; they only apply to the sentences
of some target language to which they do not belong. But they are substantial; for instance,
within the truth theory of the target system $\mathcal{S}$ we will be able to prove (a formalization
of) the statement that every theorem provable in $\mathcal{S}$ is true. Since we can also prove in this
truth theory that, say, $0 = 1$ is not true, we will then be able to derive that $0 = 1$ is not a theorem
of $\mathcal{S}$. In other words, the truth theory of $\mathcal{S}$ will always prove that $\mathcal{S}$
is consistent, which we know that $\mathcal{S}$, if it really is consistent, cannot do. This shows that truth
theories increase deductive strength.

This raises the question of what more can be achieved by iterating the construction. We would expect
to produce a hierarchy of truth theories of increasing deductive strength, a quality that can be
measured by asking what the provable ordinals of the system are. Here we say that $\alpha$ is a
{\it provable ordinal} of a system $\mathcal{S}$ if (1) there is a Turing machine that outputs (in
suitably encoded form) a well-ordering of $\mathbb{N}$ which is isomorphic to $\alpha$,
and (2) we can prove in $\mathcal{S}$ that the output of this Turing machine is a well-ordering of
$\mathbb{N}$. If we are working with first-order systems that cannot directly express the concept
of well-ordering, we take being able to prove, in $\mathcal{S}$, transfinite induction up to $\alpha$
(as encoded in $\mathbb{N}$) for all predicates in its language as a stand-in for well-ordering.
The {\it ordinal strength} of $\mathcal{S}$ is the supremum of its provable ordinals.

The subtlety here is that, given a hierarchy of truth theories, there may be a limit as to how much of
the hierarchy we trust. Intuitively, we convince ourselves that a theory in the hierarchy is sound by
a transfinite induction argument, meaning that before accepting its theorems we would need its index
to be, not just an ordinal, but a provable ordinal. As we work our way up the hierarchy, we expect to
accumulate ever-larger provable ordinals, leading us to accept ever-higher theories in the hierarchy.
But the way it is actually going to work is more subtle than this.

\section{The theory ${\rm Tarski}(\mathcal{S})$}

Throughout this paper $\mathcal{S}$ will be a first-order theory which extends Peano arithmetic.
Second-order theories could be analyzed in the same way, with trivial modifications.
Our analysis is indifferent to the use of classical or intuitionistic logic.

The simplest way to augment $\mathcal{S}$ with a truth predicate would be to
add a predicate symbol $T$, together with, for each sentence $A$ of the language of $\mathcal{S}$,
the axiom $A \leftrightarrow T(\ulcorner A\urcorner)$. This would implement
Tarski's ``T-scheme'', so that it would qualify as a truth predicate in Tarski's sense, but it
would accomplish almost nothing. We would have produced a conservative extension of the
original theory, yielding no new theorems expressible in the original, unaugmented language.
Even simple assertions like ``for all sentences $A$ and $B$, if $A$ and $B$ are both true then so
is $A \wedge B$'' could not be proven, as only finitely many instances of the T-scheme could be
used in any proof, which clearly is not sufficient to draw this conclusion.

We need to be able to reason about truth globally. So I define ${\rm Tarski}(\mathcal{S})$ to be the
formal system system whose language $L_T$ is the language $L$ of $\mathcal{S}$ augmented by a
single predicate symbol $T$, and whose nonlogical axioms are the nonlogical axioms of $\mathcal{S}$
together with

\begin{itemize}
\item a single axiom which affirms\footnote{Literally, the statement
``$(\forall n)T(f(n))$'' for some primitive recursive function $f$ which enumerates the G\"odel
numbers of the axioms of $\mathcal{S}$.}
$T(\ulcorner A\urcorner)$ for all the axioms $A$ of $\mathcal{S}$

\item a single axiom which affirms
$$T(\ulcorner A\urcorner) \wedge T(\ulcorner B\urcorner) \to T(\ulcorner C\urcorner)$$
for any $A,B,C \in L$ for which there is a deduction rule that infers $C$ from $A$ and $B$

\item for every formula $A = A(x)$ in $L_T$ with free variable $x$
(and possibly other free variables not shown), the induction axiom
$$(A(0) \wedge (\forall n)(A(n) \to A(n+1))) \to (\forall n) A(n)$$

\item the {\it $\omega$-rule}, a single axiom which states that for every predicate $A$ in $L$
(i.e., every formula in $L$ with exactly one free variable) we have\footnote{We cannot write
``$(\forall n)T(\ulcorner A(n)\urcorner)$'' here; this expression is senseless because
$\ulcorner A(n)\urcorner$ is a number and does not contain any variables that could be
quantified. We need to talk about the formula obtained by replacing
every occurrence of the variable symbol `$n$' in $A$
with a constant term that evaluates to the actual number $n$. So something like
$(\forall n)T(g(\ulcorner A\urcorner, n))$, where $g$ is a primitive recursive function
that, given the G\"odel number of a predicate, returns the G\"odel number of that predicate with every
occurence of its free variable replaced by $\hat{n}$.}
$$(\forall n)T(\ulcorner A(\hat{n})\urcorner) \leftrightarrow
T(\ulcorner (\forall n)A(n)\urcorner),$$
where $\hat{n}$ represents a canonical constant term that evaluates to $n$

\item the T-scheme, which consists of one axiom of the form
$$\overline{A} \leftrightarrow T(\ulcorner A\urcorner)$$
for each formula $A$ in $L$, where $\overline{A}$ is the universal closure of $A$.
\end{itemize}

The intention is to interpret $T(\ulcorner A\urcorner)$ as affirming the truth of $A$, or of the
universal closure of $A$ if it has any free variables.

\section{Justifying ${\rm Tarski}(\mathcal{S})$}

How do we justify accepting ${\rm Tarksi}(\mathcal{S})$, given that we accept $\mathcal{S}$?

By ``accept'' I mean something like: we understand that we have the right, indeed that we are rationally compelled,
to affirm every theorem of $\mathcal{S}$. Why should we feel the same way about ${\rm Tarksi}(\mathcal{S})$?

The question seems a little murky because ${\rm Tarski}(\mathcal{S})$ invokes the notion of ``truth'', which is
surprisingly subtle, and still the subject of philosophical debate. I have my own explanation of the nature
of truth, which I have given in detail in \cite{weaver2}, but I will not call on it here. Instead, I will argue that
a basic grasp of infinitary reasoning would allow us to accept ${\rm Tarksi}(\mathcal{S})$ without requiring
prior familiarity with any notion of truth.

My particular interest is in ``countablist'' reasoning which admits countably infinite, but not uncountable, sets and
constructions --- and sentences and derivations. This alone is enough to justify passing from $\mathcal{S}$
to ${\rm Tarski}(\mathcal{S})$, assuming that the language of $\mathcal{S}$ has only countably many formulas.
Because then we could enumerate these formulas as $A_1$, $A_2$, $\ldots$ and simply define the desired
predicate $T$ by the infinite conjunction
$$(T(\ulcorner A_1\urcorner) \leftrightarrow \overline{A}_1)\wedge
(T(\ulcorner A_2\urcorner) \leftrightarrow \overline{A}_2) \wedge \cdots.$$
No prior ideas about ``truth'' are needed
to do this. I suppose all the material in this paper could be reworked in terms of countably infinite logic,
without mentioning truth at all --- I leave this as a challenge to the reader.

The definition of $T$ I just gave clearly forces us to accept every instance of the T-scheme. What about the
other axioms of ${\rm Tarksi}(\mathcal{S})$? For instance, how do we infer ``$T(\ulcorner A\urcorner)$,
for every axiom $A$ of $\mathcal{S}$''?  This, of course, demands that we have already accepted all the
axioms of $\mathcal{S}$, and not just schematically --- we should have accepted the single statement
$\overline{B}_1 \wedge \overline{B}_2 \wedge \cdots$ where the $B_n$'s enumerate the axioms of $\mathcal{S}$. Granting this, and given the definition of $T$, we can infer $T(\ulcorner B_1\urcorner) \wedge
T(\ulcorner B_2\urcorner) \wedge \cdots$, and from there the single statement ``$T(\ulcorner A\urcorner)$,
for every axiom $A$ of $\mathcal{S}$''.

A similar justification can be given for the second axiom, which states that deduction under $T$ is valid.
For the induction scheme, which is applied to the language of ${\rm Tarski}(\mathcal{S})$, not just
of $\mathcal{S}$, we do not need to assume anything about $T$. All we have to do is imagine
being given the premises $A(0)$ and $(\forall n)(A(n) \to A(n+1))$, for some formula $A$ with
exactly one free variable, in the language of ${\rm Tarski}(\mathcal{S})$. Using countablist
reasoning, we can easily see how to construct a proof from this of $A(1)$, then a proof of $A(2)$,
etc., yielding finally (with one final infinite inference) $(\forall n)A(n)$. Any formula with more than one
free variable can then be handled by the countably long process of separately applying the preceding
reasoning to this formula with its other free variables replaced by constant terms in every possible
way.

Given our definition of $T$, the $\omega$-rule is straightforwardly justifiable along countablist lines.
This completes my sketch of a justification of ${\rm Tarski}(\mathcal{S})$, given $\mathcal{S}$.

I mentioned in the introduction that it will be a theorem of ${\rm Tarski}(\mathcal{S})$ that every
theorem of $\mathcal{S}$ is true. This is proven by induction on the number of steps in the derivation
of a theorem of $\mathcal{S}$. The theorems provable in a single step are precisely the axioms of
$\mathcal{S}$, which we know are all true, and the truth of theorems provable at any subsequent step
then follows inductively from the fact that $T$ respects the deduction rules of $\mathcal{S}$. The
assertion that $0 = 1$ is not true is an immediate consequence of the forward
implication in `$T(\ulcorner 0 = 1\urcorner) \leftrightarrow 0=1$''.

\section{The theory ${\rm Tarski}_{\preceq}(\mathcal{S})$}

Now let us iterate the ${\rm Tarski}(\mathcal{S})$ construction. Start with a computable total ordering
$\preceq$ of $\mathbb{N}$. The intention is to iterate along a well-ordering of $\mathbb{N}$, but we do not
assume that we know $\preceq$ is well-ordered. We do need to know that it is a total ordering,
that it has a least element, that every element has an immediate successor, and so on. There must be
operations of addition, multiplication, and exponentiation which correspond to the usual ordinal operations.
One can prove in Peano arithmetic that any of the standard notation systems for countable ordinals has
all the properties we need.

In the sequel, any notation for sum, product, or exponent will {\it always} refer to these (presumptively)
ordinal operations, not to the usual arithmetical operations on $\mathbb{N}$. Also, from here on, the
symbols ``$0$'', ``$1$'', etc will {\it always} refer to the $\preceq$-minimal element of $\mathbb{N}$, its
immediate $\preceq$-successor, and so on.\footnote{I will also write things like ``$A(2)$'' to mean ``the
formula obtained by replacing the free variable in $A$ with a constant term that evaluates to $2$ (as
I have already done a few paragraphs back). It might be more correct to write ``$A(\hat{2})$'', but I see
no harm in omitting the hat. Similarly, I will write ``$A(\omega)$'', not ``$A(\hat{\omega})$''. But I will
keep the hats on variable symbols, as in ``$A(\hat{n})$'', when appropriate.}
I will write ``$\omega$'' both for the actual ordinal $\omega$
and for the $\omega$th element of $\mathbb{N}$ according to $\preceq$.

The language of ${\rm Tarski}_{\preceq}(\mathcal{S})$, which I denote $L_{\preceq}$, will be the
language of $\mathcal{S}$ augmented
by a family of predicate symbols $T_a$, for $a \in \mathbb{N}$, and one additional predicate symbol
${\rm Acc}$. The idea is that the $T_a$'s should be the truth predicates in the hierarchy and that
${\rm Acc}(a)$ represents our acceptance of the hierarchy up to $a$ (with respect to the ordering
$\preceq$).

For each $a \in \mathbb{N}$ we let $L_a$ be the language of $\mathcal{S}$ augmented by only the
predicate symbols $T_b$ with $b \prec a$. It is the formulas of $L_a$ to which we want to apply
$T_a$.

We would like a hierarchy of formal systems $\mathcal{S}_a$ with $\mathcal{S}_{a + 1} =
{\rm Tarski}(\mathcal{S}_a)$ and $\mathcal{S}_a = {\rm Tarski}(\bigcup_{b \prec a} \mathcal{S}_b)$ at limits.
This could be achieved by a recursive construction, but it is simpler (and also a slightly different, better
result) to define them all at once, in the following way.

For each $a \in \mathbb{N}$ let $\mathcal{S}_a$ be the formal system whose language is $L_{a+1}$
(so it includes the predicate $T_a$) and whose nonlogical axioms are the nonlogical axioms of $\mathcal{S}$
and the induction scheme for every formula in $L_{a+1}$,  together with axioms about $T_a$:

\begin{itemize}
\item the statement ``$T_a(\ulcorner A\urcorner)$ for all the axioms $A$ of $\mathcal{S}$''

\item the statement ``$T_a(\ulcorner A\urcorner) \wedge T_a(\ulcorner B\urcorner) \to
T_a(\ulcorner C\urcorner)$ for all $A, B, C \in L_a$ for which there is a deduction rule that
infers $C$ from $A$ and $B$''

\item a single axiom which says that $T_a$ holds of every instance of the induction scheme for formulas
in $L_a$

\item a single statement that affirms
$$(\forall n)T_a(\ulcorner A(\hat{n})\urcorner) \leftrightarrow
T_a(\ulcorner (\forall n)A(n)\urcorner)$$
for every predicate $A$ in $L_a$

\item for every formula $A$ in $L_a$, the truth definition
$$\overline{A} \leftrightarrow T_a(\ulcorner A\urcorner)$$
\end{itemize}

\noindent and axioms about $T_b$ for $b \prec a$:

\begin{itemize}
\item the single statement ``for all $b \prec a$ and every $A \in L_b$ we have
$T_a(\ulcorner\overline{A} \leftrightarrow T_b(\ulcorner A\urcorner)\urcorner)$''.

\item for each $b \prec a$, the statement ``$T_b(\ulcorner A\urcorner) \leftrightarrow
T_a(\ulcorner A\urcorner)$ for all $A \in L_b$''
\end{itemize}

$\mathcal{S}_a$ is the theory whose acceptance is expressed in ${\rm Tarski}_{\preceq}(\mathcal{S})$ by
${\rm Acc}(a)$. The appropriate axiom for ${\rm Acc}(\cdot)$ is that it should be {\it progressive}, i.e.,
$$(\forall b \prec a){\rm Acc}(b) \leftrightarrow {\rm Acc}(a).$$
I will write ``${\rm Prog}({\rm Acc})$'' for this condition. (Note that it is stronger than the usual notion of
progressivity, which affirms only the forward implication. This choice simplifies the analysis given below a little.)
Indeed, $\mathcal{S}_0$ is precisely ${\rm Tarksi}(\mathcal{S})$, and for any $a \succ 0$, although
$\mathcal{S}_a$ is not exactly the same as ${\rm Tarksi}(\bigcup_{b \prec a} \mathcal{S}_b)$
(because of the final ``$T_b(\ulcorner A\urcorner) \leftrightarrow
T_a(\ulcorner A\urcorner)$ for all $A \in L_b$'' scheme),
essentially the same argument can be used to justify accepting it given that we have accepted
$\bigcup_{b \prec a} \mathcal{S}_b$. So there is a straightforward justification of progressivity of ${\rm Acc}$.

It is easy to see that for every $b \prec a$ it is a theorem of $\mathcal{S}_a$ that all the axioms of
$\mathcal{S}_b$ are true. Together with the truth of all deduction rules for formulas in $L_a$, this yields
the following fact.

\begin{prop}\label{provetrue}
For every $b \prec a$, it is a theorem of $\mathcal{S}_a$ that every theorem of $\mathcal{S}_b$ is true.
\end{prop}

Now I am ready to describe the axioms of ${\rm Tarski}_{\preceq}(\mathcal{S})$. Its language was
defined earlier; its nonlogical axioms will consist of

\begin{itemize}
\item the nonlogical axioms of $\mathcal{S}$

\item the axiom ${\rm Prog}({\rm Acc})$

\item the induction scheme for all formulas in the language of $\mathcal{S}$ augmented by ${\rm Acc}$

\item for each $a$ and each axiom $A$ of $\mathcal{S}_a$ the statement ``${\rm Acc}(\hat{a}) \to A$''.
\end{itemize}

\noindent Since the meanings of the $T_a$ predicates are not initially specified, I exclude them
from the induction scheme. For the record, I think this precaution is philosophically unnecessary, but for my
purposes there is no harm in being conservative, as none of the results proven below is affected by this restriction.
We could be even more conservative and forbid any reasoning about formulas containing $T_a$ until
${\rm Acc}(\hat{a})$ has been proven by, for each $a$ and each axiom $A$ of $\mathcal{S}_a$, replacing
the ``${\rm Acc}(\hat{a}) \to A$'' axiom with a deduction rule that infers $A$ from ${\rm Acc}(\hat{a})$.
This also would not affect the validity of anything that follows in any essential way.

\section{The Veblen hierarchy}

We can now start to think about the provable ordinals of ${\rm Tarski}_{\preceq}(\mathcal{S})$. Let us
do this in terms of the Veblen hierarchy.

The {\it Veblen functions} $\phi_\alpha$ map ordinals into ordinals, and are defined by induction on the
ordinal $\alpha$ by setting $\phi_0(\beta) = \omega^\beta$ and, for all $\alpha > 0$, letting $\phi_\alpha$
enumerate the common fixed points of $\{\phi_{\alpha'}: \alpha' < \alpha\}$. At a successor stage
$\alpha + 1$ this just means that $\phi_{\alpha + 1}$ enumerates the fixed points of $\phi_\alpha$.

We have fixed a computable total ordering $\preceq$ of $\mathbb{N}$. Now fix an element
$c_0 \in \mathbb{N}$ which is greater than $\omega$ and is stable under ordinal exponentiation, i.e.,
if $a,b \prec c_0$ then $a^b \prec c_0$. Again, we don't have to know that $\{c \in \mathbb{N}: c \prec c_0\}$
is well-ordered. Recall that a predicate $A$ is progressive if for all $a \in \mathbb{N}$
$$(\forall b \prec a)A(b) \leftrightarrow A(a);$$
I will also say it is {\it progressive up to $c_0$} if $(\forall b \prec a)A(b) \leftrightarrow A(a)$
for all $a \prec c_0$.

For each $a,c \prec c_0$ define a predicate $B_a^c$ in the language $L_{\omega^a\cdot c + 1}$ by
$$B_a^c(b) := (\forall A \in L_{\omega^a\cdot c}^p)
T_{\omega^a\cdot c}(\ulcorner{\rm Prog}(A) \to A(\phi_a(b))\urcorner)$$
where $L_a^p$ denotes the set of predicates in $L_a$. For fixed $a$, the predicates $B_a^c$ say
essentially the same thing as $c$ varies: that any progressive predicate holds of $\phi_a(b)$. The only
difference is the languages to which they are
applied. This is necessary because we want to apply some $B_a^c$'s to others.

Let $C$ be the predicate in the language $L_{c_0+1}$
$$C(a) := (\forall c \prec c_0)T_{c_0}(\ulcorner {\rm Prog}(B_a^c)\urcorner)).$$
Recalling that $B_a^c$ and $B_a^{c'}$ differ only
in the languages to which they are applied, the crude intuition for $C(a)$ is that it says the assertion
that anything progressive holds of $\phi_a(b)$ is itself progressive in $b$. The key theorem we have to
prove is that $C$ is itself progressive!

(The reason the strong version of progressivity defined above helps is that, under this definition, if
$A$ is progressive and $A$ holds of $a$, then the ``shifted'' predicate $A^{\to a}(\cdot) =
A(a + \cdot)$ is also progressive. Our version of progressivity ensures that if $A(a)$ fails
then $A(a')$ also fails for all $a' \succ a$, so this shifting can only occur within the initial,
truly progressive part.)

The proof of Theorem \ref{progofc} below is a little long because each $a$ and each $b$ can be zero,
a successor, or a limit, leading to nine separate cases which are nearly all genuinely different. (For limit
values of $b$, the argument is essentially the same regardless of what $a$ is.)
For the sake of readability I have split up the proof into three lemmas.
The crucial case appears in Lemma \ref{lemma3}, when $a$ and $b$ are both successors. (It is
only here that the ``$\omega^a\cdot c$'' expression in $B_a^c$ becomes important. The point is
that the index $\omega^{a+1}\cdot c$ will always be a limit of indices of the form $\omega^a\cdot c'$.)

In the following proofs we will assume ${\rm Acc}(\hat{c}_0)$, which means that all the axioms for
$T_{c_0}$ are available. In particular, we have the T-scheme for $T_{c_0}$. So we can pass freely
between $\overline{A}$ and $T_{c_0}(A)$. What we cannot do is to reason under $T_{c_0}$ about
predicates such as $C$ which do not belong to $L_{c_0}$. Note that in the definition of $C(a)$ given above
we cannot use the
$\omega$-rule to import the quantifier $\forall c \prec c_0$ into $T_{c_0}$, because $B_a^c$ is a
distinct predicate for every $a$ and $c$, so an expression like ``$(\forall c \prec c_0){\rm Prog}(B_a^c)$''
would not be a well-formed formula. But if we are working under $T_{c_0}$, we can simultaneously
reason about $B_a^c$ for arbitrary values of $c$. Indeed, this is just the sort of thing truth predicates
are good for --- allowing us to reason schematically. I have included some
comments about which truth predicate we are reasoning under at which point, but the
essence of the proofs can be understood by ignoring all references to truth and just accepting
that we can reason schematically.

The next three lemmas are proven in ${\rm Tarski}_{\preceq}(\mathcal{S})$.

\begin{lemma}\label{lemma1}
${\rm Acc}(\hat{c}_0)$ implies $C(0)$.
\end{lemma}

\begin{proof}
Working in ${\rm Tarski}_{\preceq}(\mathcal{S})$, assume ${\rm Acc}(\hat{c}_0)$. Fix $c \prec c_0$;
reasoning under $T_{c_0}$, we must prove
${\rm Prog}(B_0^c)$, i.e., we must prove progressivity in $b$ of the statement
\begin{equation}\label{firstone}
(\forall A \in L_c^p)T_c(\ulcorner{\rm Prog}(A) \to A(\omega^b)\urcorner).
\end{equation}

For $b = 0$ this is trivial; it just says that for any $A \in L_c^p$ it is true (according to $T_c$)
that  progressivity of $A$ implies that $A$ holds of $\omega^0 = 1$. The way we prove this in
${\rm Tarski}_{\preceq}(\mathcal{S})$ is by first proving in ${\rm Tarski}_{\preceq}(\mathcal{S})$
--- what can actually be done in Peano arithmetic --- a formalization of the statement that for any
predicate $A$ in $L_c$ there is a proof in the system $\mathcal{S}_c$ from Section 4 of the
sentence ``${\rm Prog}(A) \to A(1)$''. We can do this by writing out a proof template and then
verifying its validity for any $A$. Combining this with Proposition \ref{provetrue} (and using
${\rm Acc}(\hat{c}_0)$), we get $(\forall A \in L_c^p)T_{c_0}(\ulcorner{\rm Prog}(A) \to A(1)\urcorner)$,
and then $(\forall A \in L_c^p)T_c(\ulcorner{\rm Prog}(A) \to A(1)\urcorner)$.
I wanted to spell all this out once, but I will omit similar arguments in the sequel.

For limit values of $b$, assume we are given $(\forall b' \prec b)T_c(B_0^c(b'))$; that is, assume
$$T_c(\ulcorner {\rm Prog}(A) \to A(\omega^{b'})\urcorner)$$
for all $A \in L_c^p$ and all $b' \prec b$. Using the $\omega$-rule to bring the quantifier
``$\forall b' \prec b$'' under $T_c$, we get that for every $A \in L_c^p$ it is true (according to
$T_c$) that progressivity of $A$ implies $A(\omega^{b'})$ for all $b' \prec b$. Since $\omega^b =
\sup_{b' \prec b} \omega^{b'}$, we infer from this (for arbitrary $A$, under $T_c$) that
progressivity of $A$ implies $A(\omega^b)$.

Finally, if $b = b' + 1$ is a successor, assume that for every predicate $A$ in $L_c$ it is true
(according to $T_c$) that progressivity of $A$ implies that $A$ holds of $\omega^{b'}$.
Since ${\rm Prog}(A) \wedge A(\omega^{b'})$ implies ${\rm Prog}(A^{\to \omega^{b'}})$ (as
noted a few paragraphs before this lemma), it is then true that $A^{\to \omega^{b'}}$ holds
of $\omega^{b'}$, i.e., $A$ holds of $\omega^{b'} + \omega^{b'}$. Continuing inductively, we get the
truth for all $n \prec \omega$ of $A(\omega^{b'}\cdot n)$, and progressivity of $A$ then
implies that $A$ holds of $\omega^{b'}\cdot \omega = \omega^{b' + 1} = \omega^b$.
We have shown that $B_0^c(b')$ implies $B_0^c(b' + 1)$. This
completes the proof that (\ref{firstone}) is progressive in $b$ and establishes $C(0)$.
\end{proof}

\begin{lemma}\label{lemma2}
${\rm Acc}(\hat{c}_0)$ implies that for any limit $a \prec c_0$ we have
$(\forall a' \prec a)C(a') \to C(a)$.
\end{lemma}

\begin{proof}
Working in ${\rm Tarski}_{\preceq}(\mathcal{S})$, assume ${\rm Acc}(\hat{c}_0)$.
Fix a limit $a \prec c_0$ and assume $(\forall a' \prec a)C(a')$. This means that we are given the truth,
according to $T_{c_0}$, of ${\rm Prog}(B_{a'}^{c'})$ for every $a' \prec a$ and $c' \prec c_0$, and we
must show that for every $c \prec c_0$ it is true according to $T_{c_0}$ that $B_a^c$ is progressive.

Fix $c \prec c_0$. First we check the case $b = 0$, where we have to
verify $B_a^c(0)$. Since we are assuming, for every $c' \prec c_0$ and $a' \prec a$, that it is true
(according to $T_{c_0}$) that $B_{a'}^{c'}$ is progressive, in particular
we have for all such $c'$ and $a'$ the truth of $B_{a'}^{c'}(0)$. Thus, for every predicate
$A$ in some $L_{\omega^{a'}\cdot c'}$ with $a' \prec a$ and $c' \prec c_0$ --- which is to say, every
predicate $A$ in $L_{c_0}$ --- we have that progressivity of $A$ implies
the truth of $A(\phi_{a'}(0))$ for all $a' \prec a$. In particular, for any predicate
$A \in L_{\omega^a\cdot c}$ it is true according to $T_{c_0}$, and hence according to
$T_{\omega^a\cdot c}$, that
progressivity of $A$ implies that $A$ holds of $\sup_{a' \prec a} \phi_{a'}(0) = \phi_a(0)$.
This verifies $B_a^c(0)$. That was the $b = 0$ case.

Next we will show, reasoning under $T_{c_0}$, that if $b$ is a limit and we have $B_a^c(b')$
for all $b' \prec b$, then we also have $B_a^c(b)$. At this point
we have two induction hypotheses going, one for $a' \prec a$ and one for $b' \prec b$.
To verify $B_a^c(b)$, let $A$ be a predicate in $L_{\omega^a\cdot c}$. Then
$B_a^c(b')$ for all
$b' \prec b$ tells us that according to $T_{\omega^a\cdot c}$, progressivity of $A$ implies that
$A$ holds of $\phi_a(b')$ for all $b' \prec b$. So according to $T_{\omega^a\cdot c}$,
progressivity of $A$
implies that $A$ holds of $\sup_{b' \prec b} \phi_a(b') = \phi_a(b)$, as desired. We
did not even need the induction hypothesis on $a$ in this part.

The final case is for successor values of $b$. Say $b = b' + 1$. Reasoning under $T_{c_0}$, we
have to verify that $B_a^c(b')$ implies $B_a^c(b' + 1)$. For this we must use both
(i) the induction hypothesis on $a$ (for every $c' \prec c_0$ and $a' \prec a$ the predicate
$B_{a'}^{c'}$ is progressive) and (ii) the induction hypothesis on $b$ (we have $B_a^c(b')$).
Now $B_a^c(b')$ says that every progressive $A \in L_{\omega^a\cdot c}$ holds of $\phi_a(b')$,
and $\phi_a(b')$ is a fixed point of every $\phi_{a'}$ with $a' \prec a$, so we can also say
that every progressive $A \in L_{\omega^a\cdot c}$ holds of $\phi_{a'}(\phi_a(b'))$ for
all $a' \prec a$. This means that $B_{a'}^{c'}(\phi_a(b'))$ holds for all $a' \prec a$ and
$c' \prec c_0$ such that $L_{\omega^{a'}\cdot c'} \subseteq L_{\omega^a\cdot c}$.

But for any $a' \prec a$, we can find $a''$ such that $a' + a'' = a$, and then by putting
$c' = \omega^{a''}\cdot c$, we get $\omega^{a'}\cdot c' = \omega^a\cdot c$. So by the
conclusion reached in the last paragraph, for any $a' \prec a$ there exists $c' \prec c_0$
such that $L_{\omega^{a'}\cdot c'} = L_{\omega^a\cdot c}$ and $B_{a'}^{c'}(\phi_a(b'))$
holds. And since every such $B_{a'}^{c'}$ is progressive,
by (i), it follows that $B_{a'}^{c'}(\phi_a(b') + 1)$ also holds.

We are now in a position to affirm that for every $a' \prec a$, every progressive
$A \in L_{\omega^a\cdot c}$ lies within the scope of some $B_{a'}^{c'}$ which satisfies
$B_{a'}^{c'}(\phi_a(b') + 1)$, so that $A$ holds of $\phi_{a'}(\phi_a(b') + 1)$. If $a'$ is a
successor, $a' = a'' + 1$, then this means that $A$ holds of the next fixed point of
$\phi_{a''}$ after $\phi_a(b')$. Thus, by progressivity, $A$ holds of the supremum of
these fixed points over all $a'' \prec a$, which equals $\phi_a(b'+1) = \phi_a(b)$. That
verifies $B_a^c(b)$ and completes the proof of $C(a)$.
\end{proof}

\begin{lemma}\label{lemma3}
${\rm Acc}(\hat{c}_0)$ implies that for any $a \prec c_0$, $C(a)$ implies $C(a + 1)$.
\end{lemma}

\begin{proof}
Working in ${\rm Tarski}_{\preceq}(\mathcal{S})$, assume ${\rm Acc}(\hat{c}_0)$.
Fix $a \prec c_0$ and assume $C(a)$, and on the way to proving $C(a+1)$, fix $c \prec c_0$.
Reasoning under $T_{c_0}$, we must prove that $B_{a+1}^c$ is progressive, i.e., that for every
$b$, if $B_{a+1}^c(b')$ holds for all $b' \prec b$ then $B_{a+1}^c(b)$ also holds.

If $b = 0$ then we simply have to show that $B_{a+1}^c(0)$ holds, i.e., that any progressive
$A \in L_{\omega^{a+1}\cdot c}$ holds of $\phi_{a+1}(0)$. Since we have assumed $C(a)$, we
know that for any $c' \prec c_0$, according to $T_{c_0}$ the predicate $B_a^{c'}$ is progressive. In
particular, for each $n \prec \omega$ the predicate $B_a^{c + n}$ is progressive. Thus $B_a^{c+n}$
holds of $0$, so every progressive predicate in $L_{\omega^a(c + n)}$ holds of $\phi_a(0)$,
and hence of $\phi_a(0) + 1$. In particular, $B_a^{c + n - 1}$ holds of $\phi_a(0) + 1$.
So every progressive predicate in $L_{\omega^a(c+n-1)}$ holds of $\phi_a(\phi_a(0) + 1)$.
Inductively, every progressive predicate in $L_{\omega^a(c + n - k)}$ holds of
$\phi_a^k(\phi_a(0) + 1)$, which means (taking $k = n$) that every progressive
predicate in $L_{\omega^a\cdot c}$ holds of $\phi_a^n(\phi_a(0) + 1)$. Since $n$ was arbitrary, it
follows that every progressive predicate in $L_{\omega^a\cdot c}$ holds of $\sup_{n \prec \omega}
\phi_a^n(\phi_a(0) + 1) = \phi_{a+1}(0)$. This finishes the $b = 0$ case.

Next, assuming that for some limit $b$ we have $B_{a+1}^c(b')$
for all $b' \prec b$, we must verify $B_{a+1}^c(b)$. This is easy, because we are given that any
progressive predicate in $L_{\omega^{a+1}\cdot c}$ holds of $\phi_{a+1}(b')$ for all $b' \prec b$,
and so by progressivity it holds of $\sup_{b' \prec b}\phi_{a+1}(b') = \phi_{a+1}(b)$. The corresponding
cases in Lemmas \ref{lemma1} and \ref{lemma2} were handled in the same way.

Finally, letting $b$ be arbitrary, we must
make the inference from $B_{a+1}^c(b)$ to $B_{a+1}^c(b + 1)$. So assume $B_{a+1}^c(b)$ and fix
$A \in L^p_{\omega^{a+1}\cdot c}$. We can write $\omega^{a+1}\cdot c = \omega^a\cdot \omega c$,
and since $\omega c$ is a limit, this shows that $A$ must belong to $L_{\omega^a\cdot c'}$ for
some $c' \prec \omega c$. Moreover, since $\omega c$ is a limit, we have $c' + n \prec \omega c$
for all $n \prec \omega$.

Now for each $n \prec \omega$ the
predicate $B_a^{c' + n}$ belongs to $L_{\omega^{a+1}\cdot c}$, and is progressive since we are
assuming $C(a)$, so the hypothesis of $B_{a+1}^c(b)$ yields $B_a^{c' + n}(\phi_{a+1}(b))$,
and then by progressivity we get $B_a^{c' + n}(\phi_{a+1}(b) + 1)$. Applying this to
$B_a^{c' + n - 1}$, which lies in $L_{\omega^a(c' + n - 1) + 1} \subseteq L_{\omega^a(c + n)}$, then yields
$B_a^{c' + n - 1}(\phi_a(\phi_{a+1}(b) + 1))$,
and inductively we finally get $B_a^{c'}(\phi_a^n(\phi_{a+1}(b) + 1))$. Since
$A \in L_{\omega^a\cdot c'}$, if $A$ is progressive this entails
$A(\phi_a^{n+1}(\phi_{a+1}(b) + 1))$, and as $n$ was arbitrary
$A(\phi_{a+1}(b + 1))$ follows. We have established $B_{a+1}^c(b)$, as desired.
\end{proof}

Putting these lemmas together yields the following theorem.

\begin{theo}\label{progofc}
${\rm Tarski}_{\preceq}(\mathcal{S})$ proves that ${\rm Acc}(\hat{c}_0)$
implies $C$ is progressive up to $c_0$.
\end{theo}

Define $\gamma_0 = 0$, $\gamma_1 = \phi_0(0) = 1$, $\gamma_2 = \phi_1(0) =
\varepsilon_0$, etc., with $\gamma_{n+1} = \phi_{\gamma_n}(0)$. The ordinal $\Gamma_0$
is defined as the limit $\Gamma_0 = \sup_n \gamma_n$. Just as with $\omega$, I will use these
symbols both for abstract ordinals and their numerical representatives in $(\mathbb{N}, \preceq)$.

In what follows the term ``transfinite induction up to $a$ for $\mathcal{S}$'' will refer to the scheme
consisting of all sentences of the form ``${\rm Prog}_{\hat{a}}(A) \to (\forall b \prec \hat{a})A(b)$'' as $A$ ranges
over the predicates in the language of $\mathcal{S}$, where ``${\rm Prog}_{\hat{a}}(A)$'' means ``$A$ is progressive
up to $\hat{a}$''.

\begin{coro}\label{upgamma}
Let $n \geq 2$. Then ${\rm Tarski}_{\preceq}(\mathcal{S}))$ plus transfinite induction up to $\gamma_n$
for ${\rm Tarski}_{\preceq}(\mathcal{S})$ proves transfinite induction up to $\gamma_{n+1}$ for $\mathcal{S}$.
\end{coro}

\begin{proof}
Fix $n$. Since ${\rm Acc}$ is progressive, transfinite induction up to $\gamma_n$ for
${\rm Tarski}_{\preceq}(\mathcal{S})$ yields $(\forall a \prec \gamma_n){\rm Acc}(a)$, and then,
with one further application of progressivity, ${\rm Acc}(\gamma_n)$. Theorem \ref{progofc} then
yields, in ${\rm Tarski}_{\preceq}(\mathcal{S})$ plus ${\rm Acc}(\gamma_n)$, that $C$ is progressive
up to $\gamma_n$, and now the hypothesis on transfinite induction up to $\gamma_n$ yields
$(\forall a \prec \gamma_n)C(a)$. In particular, we have $(\forall a \prec \gamma_n)T_{c_0}(B_a^0(0))$.

Thus for any predicate $A$ in $L$, we have
$$T_{c_0}({\rm Prog}(A) \to (\forall a \prec \gamma_n)A(\phi_a(0))),$$
hence
$${\rm Prog}(A) \to (\forall a \prec \gamma_n)A(\phi_a(0))$$
hence
$${\rm Prog}(A) \to (\forall b \prec \gamma_{n+1})A(b)$$
since $\sup_{a \prec \gamma_n} \phi_a(0) = \phi_{\gamma_{n+1}}$. We have proven transfinite
induction up to $\gamma_{n+1}$ for $\mathcal{S}$.
\end{proof}

\section{Iterating the ${\rm Tarski}_{\preceq}(\mathcal{S})$ construction}

Theorem \ref{progofc} doesn't get us very far by itself, since it requires a strong assumption about
${\rm Acc}$. What we must do now is to ``reflect'' on the legitimacy of the
${\rm Tarski}_{\preceq}$ construction. We have agreed that, for arbitrary $\mathcal{S}$, if we
accept $\mathcal{S}$ then we should accept ${\rm Tarski}_{\preceq}(S)$, and so we should also accept
${\rm Tarski}^2_{\preceq}(\mathcal{S}) = {\rm Tarski}_{\preceq}({\rm Tarski}_{\preceq}(\mathcal{S}))$,
then ${\rm Tarski}^3_{\preceq}(\mathcal{S}) = {\rm Tarski}_{\preceq}({\rm Tarski}^2_{\preceq}(\mathcal{S}))$,
and so on. We could now formulate a ``higher order'' construction that iterates the $\mathcal{S} \mapsto
{\rm Tarski}_{\preceq}(\mathcal{S})$ construction along $\preceq$, which would be justified by countablist
reasoning in the same way we justified ${\rm Tarski}_{\preceq}(\mathcal{S})$. This was done in \cite{weaver},
and the ordinal strength that results from doing this can be found there. Here I only want to consider
${\rm Tarski}^\omega_{\preceq}(\mathcal{S})$.

\begin{theo}
${\rm Tarski}^\omega_{\preceq}({\rm PA}) = \bigcup_{n \prec \omega}
{\rm Tarski}^n_{\preceq}({\rm PA})$ proves ${\rm Prog}(A) \to A(\hat{a})$ for every $a \prec \Gamma_0$
and every predicate $A$ in its language.
\end{theo}

\begin{proof}
We make the argument for any first-order theory $\mathcal{S}$ that extends ${\rm PA}$.
Fix $n \prec \omega$. The main idea is to prove, in
${\rm Tarski}^n_{\preceq}(\mathcal{S})$, transfinite induction up to $\gamma_2$ for
${\rm Tarski}^{n-1}_{\preceq}(\mathcal{S})$, and then inductively apply Corollary \ref{upgamma} to get
transfinite induction up to $\gamma_{k+1}$ in ${\rm Tarski}^{n-k}(\mathcal{S})$, for $1 \leq k \leq n$,
yielding finally transinite induction up to $\gamma_{n+1}$ for $\mathcal{S}$. At that point the full
statement of the theorem will follow easily.

The base case of transfinite induction up to $\gamma_2$ is similar to Lemma \ref{lemma1}, but
different enough to merit being written out. Taking $\mathcal{S}' = {\rm Tarski}^{n-1}_{\preceq}(\mathcal{S})$,
we want to prove, in ${\rm Tarski}^n_{\preceq}(\mathcal{S}) = {\rm Tarski}_{\preceq}(\mathcal{S}')$, transfinite induction up to $\gamma_2 =\varepsilon_0$ for $\mathcal{S}'$. To do this we introduce a
{\it jump predicate} $J_A$, for every predicate $A$ in the language of $\mathcal{S}'$, defined by
$$J_A(b) := (\forall a)(A(a) \to A(a + \omega^b)).$$
This is a different predicate for each $A$. Now, for any $A$,
we start by proving in $\mathcal{S}'$ that progressivity of $A$ implies progressivity of $J_A$. Fix $A$ and
assume it is progressive. Then $J_A(0)$ says that $(\forall a)(A(a) \to A(a + 1))$, and this
is an immediate consequence of progressivity of $A$. If $b$ is a limit, then $J_A(b')$ for all
$b' \prec b$ says that for any $a$ we have $A(a) \to A(a + \omega^{b'})$ for all $b' \prec b$, and this
yields $A(a) \to A(a + \omega^b)$ for all $a$, by progressivity of $A$ again. Finally, assuming $J_A(b)$,
we must prove $J_A(b+1)$. Here we use the reasoning that $J_A(b)$ tells us that $A(a)$ implies
$A(a + \omega^b)$, and also that $A(a + \omega^b)$ implies $A(a + \omega^b + \omega^b)$, and so on;
that is, an induction argument shows that $J_A(b)$ implies $(\forall a)(A(a) \to A(a + \omega^b\cdot n))$
for all $n \prec \omega$. One final appeal to progressivity of $A$ then yields $J_A(b) \to J_A(b+1)$.

For a given predicate $A$ in the language of $\mathcal{S}'$, we have shown, in $\mathcal{S}'$, that
${\rm Prog}(A) \to {\rm Prog}(J_A)$. Now, reasoning in the one-step truth theory
${\rm Tarski}(\mathcal{S}')$, we can see that $T({\rm Prog}(A))$ implies $T({\rm Prog}(J_A))$ implies
$T({\rm Prog}(J_{J_A}))$ implies $\cdots$, for any $A$. Thus, if $A$ is progressive then $A(1)$ is true,
$J_A(1)$ is true (which implies that $A(\omega)$ is true), $J_{J_A}(1)$ is true (which implies that
$J_A(\omega)$ is true, which implies that $A(\omega^\omega)$ is true), and so on. Inductively we
get $(\forall n \prec \omega)T(A(\omega^{(n)}))$, where $\omega^{(1)} = \omega$ and $\omega^{(n+1)} =
\omega^{\omega^{(n)}}$. So $T({\rm Prog}(A))$ implies $T(A(\varepsilon_0))$.

That was in ${\rm Tarski}(\mathcal{S}')$. But in ${\rm Tarksi}_{\preceq}(\mathcal{S}')$ we have
${\rm Prog}({\rm Acc})$, so in particular we have ${\rm Acc}(0)$, which means that we can reason in
the one-step truth theory ${\rm Tarski}(\mathcal{S}')$. Thus ${\rm Tarski}^n_{\preceq}(\mathcal{S})$
proves transfinite induction up to $\gamma_2 = \varepsilon_0$ for
${\rm Tarski}^{n-1}_{\preceq}(\mathcal{S})$.

From this point, we can apply Corollary \ref{upgamma} $n - 1$ times to obtain transfinite induction up
to $\gamma_{n + 1}$ for $\mathcal{S}$, where $n$ was arbitrary. But
${\rm Tarksi}^\omega_{\preceq}(\mathcal{S})$
literally equals ${\rm Tarski}^\omega_{\preceq}({\rm Tarski}^k_{\preceq}(\mathcal{S}))$ for any $k$, so
this argument actually yields transfinite induction up to $\gamma_{n+1}$ for
${\rm Tarski}^k_{\preceq}(\mathcal{S})$, for any $k$ and any $n$.
We conclude that ${\rm Tarski}^\omega_{\preceq}(\mathcal{S})$ proves ${\rm Prog}(A) \to A(\gamma_n)$
for all $n$ and all predicates $A$ in its language.
\end{proof}

\section{Conclusion}

So predicativists, or really anyone who accepts countablist reasoning, can affirm transfinite induction up to
anything less than $\Gamma_0$ in ${\rm Tarski}^\omega_\preceq({\rm PA})$. But then one further step to
${\rm Tarski}^{\omega + 1}_\preceq({\rm PA})$ gets
us beyond $\Gamma_0$, and we can keep going. Indeed, the ``higher order''
construction mentioned in Section 6 gets vastly beyond $\Gamma_0$.

The celebrated Feferman-Sch\"utte analysis which identified $\Gamma_0$ as the exact limit of predicative
reasoning is utter nonsense --- I am sorry I have to say this so plainly, but apparently I do. It was already
thoroughly refuted in \cite{weaver}; see \cite{weaver1} for a short
explanation of the problems. This analysis
was based on the idea from Section 1 about accumulating ever-larger provable ordinals
as we work our way up a hierarchy, leading us to accept ever-higher theories in the hierarchy. The (well, one)
fatal flaw in this approach is that at stage $n$ we do not prove transfinite induction up to $\gamma_n$ for all
predicates in the language, only for predicates of the form ``$x \in A$'', i.e., we prove induction up to
$\gamma_n$ for sets. But at stage $n+1$ we require, not induction up to
$\gamma_n$ for sets, but recursion up to $\gamma_n$. So we have to make an inference from
induction to recursion, which is something predicativists cannot do. (Or, if they magically somehow could,
then they could use this fact to
get beyond $\Gamma_0$. The only way to land exactly on $\Gamma_0$ is by postulating
that predicativists can always, somehow, infer recursion from induction, but {\it they are not aware that
they have this general ability} and so cannot formulate it as a general principle. The foundations community
is going to have to choose between affirming this absurdity, which has no rational basis whatever, and admiting
that Feferman and Sch\"utte simply got it wrong.)

The induction/recursion issue doesn't arise in the analysis of truth theories presented here, because we
actually prove full induction at each stage, not just induction for sets. But the idea of working up a hierarchy
is more subtle in this approach, because one actually works {\it down} a tower of ${\rm Tarski}_{\preceq}$
theories in order to move up through the $\gamma_n$'s.

It would be very interesting to see how much further one can get in a system that formalizes the idea that
whenever we have a construction that converts any accepted system $\mathcal{S}$ into another, stronger
accepted system, we can accept a system that iterates this construction along $\preceq$. Another challenge
for the reader, one that would get us closer to understanding the {\it true} ordinal strength of predicativism.

\end{document}